\documentclass[11pt,a4paper]{article}

\usepackage{graphicx,latexsym,euscript,makeidx,color,bm}
\usepackage{amsmath,amsfonts,amssymb,amsthm,thmtools,mathtools,mathrsfs,enumerate}

\usepackage{soul}
\usepackage[colorlinks,linkcolor=blue,citecolor=red]{hyperref}

\usepackage{geometry}
\allowdisplaybreaks[4]
\def\dbE{\mathbb{E}}     
\def\dbF{\mathbb{F}}   \def\cF{{\cal F}}  
     
\def\dbH{\mathbb{H}}     
     
   \def\cJ{{\cal J}}  
     
   \def\cL{{\cal L}}

\def\dbP{\mathbb{P}}   
   
\def\dbR{\mathbb{R}}   \def\cR{{\cal R}}
\def\dbS{\mathbb{S}}   
   \def\cT{{\cal T}}
 \def\sU{\mathscr{U}}

\def\ns{\noalign{\ss}}
\def\ds{\displaystyle}
\def\ss{\smallskip}             \def\hb{\hbox}
\def\ms{\medskip}              
          \def\lan{\langle}       \def\les{\leqslant}
\def\ges{\geqslant}

\def\q{\quad}         \def\ran{\rangle}       \def\tr{\hb{tr$\,$}}
\def\qq{\qquad}             
\def\no{\noindent}          
         
\def\nn{\nonumber}         
\def\rf{\eqref}        \def\ds{\displaystyle}
\def\cd{\cdot}         
\def\deq{\triangleq}  \def\({\Big(}           
\def\ti{\tilde}       \def\){\Big)}           
   \def\[{\Big[}           
  \def\]{\Big]}
\def\({\Big (}
\def\){\Big )}
\def\[{\Big[}
\def\]{\Big]}
\def\3n{\negthinspace \negthinspace \negthinspace }
\def\2n{\negthinspace \negthinspace }
\def\1n{\negthinspace }
\def\bel{\begin{equation}\label}
\def\ee{\end{equation}}
\def\bea{\begin{eqnarray}}
\def\eea{\end{eqnarray}}
\def\bt{\begin{theorem}\label}
\def\et{\end{theorem}}
\def\bc{\begin{corollary}\label}
\def\ec{\end{corollary}}
\def\bex{\begin{example}\label}
\def\ex{\end{example}}
\def\bl{\begin{lemma}\label}
\def\el{\end{lemma}}
\def\bp{\begin{proposition}\label}
\def\ep{\end{proposition}}
\def\br{\begin{remark}\label}
\def\er{\end{remark}}
\def\ba{\begin{array}}
\def\ea{\end{array}}
\def\bde{\begin{definition}\label}
\def\ede{\end{definition}}

       \def\l{\lambda}    
               
\def\d{\delta}           \def\G{\varGamma}
\def\e{\varepsilon}      
\def\f{\varphi}             \def\Om{\varOmega}
       \def\si{\sigma}    \def\Si{\varSigma}
      \def\Th{\varTheta}

\newtheoremstyle{thry}
{}      
{}      
{\sl}   
{}      
{\bf}   
{.}     
{.5em}  
{}      
\theoremstyle{thry}

\newtheorem{theorem}{Theorem}[section]
\newtheorem{proposition}[theorem]{Proposition}
\newtheorem{corollary}[theorem]{Corollary}
\newtheorem{lemma}[theorem]{Lemma}

\theoremstyle{definition}
\newtheorem{definition}[theorem]{Definition}
\newtheorem{example}[theorem]{Example}

\theoremstyle{remark}
\newtheorem{remark}[theorem]{Remark}

\def\punct{}
\newtheoremstyle{dotless}{}{}{\rm}{}{\bf}{\punct}{.5em}{}
\theoremstyle{dotless}

\newtheorem{taggedassumptionx}{}
\newenvironment{taggedassumption}[1]
 {\renewcommand\thetaggedassumptionx{#1}\taggedassumptionx}
 {\endtaggedassumptionx}

\makeatletter
   
   \@addtoreset{equation}{section}
   \newcommand{\setword}[2]{%
   \phantomsection
   #1\def\@currentlabel{\unexpanded{#1}}\label{#2}%
   }
\makeatother

\begin{document}

\title{\bf Sub-Infinite Horizon Stochastic Linear-Quadratic Optimal Control Problems and Delayed Backward Riccati Equations}

\author{
Yutao Chen\thanks{ School of Mathematical Sciences,
Fudan University, Shanghai, 200433, China (Email: {\tt ytchen22@m.fudan.edu.cn}).}
                   ~~~
Hongwei Lou\thanks{ School of Mathematical Sciences, and LMNS
Fudan University, Shanghai, 200433, China (Email: {\tt hwlou@fudan.edu.cn}).}
                   ~~~
Hanxiao Wang\thanks{Corresponding author. School of Mathematical Sciences, Shenzhen University, Shenzhen,
                   518060, China (Email: {\tt hxwang@szu.edu.cn}). This author is supported in part by   NSFC grant 12522121.
                 The authors would like to thank Professor Jiongmin Yong for his valuable suggestions.}
}

\maketitle

\ms

\no{\bf Abstract.}
In this paper, we investigate a class of so-called sub-infinite horizon stochastic linear-quadratic  optimal control problems,
in which the initial time $t$ is arbitrarily taken from $[0,\infty)$ and the running cost is defined over $[t,t+T]$ for a given $T>0$.
The optimal control of this type of problem can be obtained by standard methods;
however, it  is shown that the resulting optimal control is generally time-inconsistent.
Thus, instead of seeking an optimal control, which is time-inconsistent,
we aim to find a time-consistent,  locally optimal, and time-invariant equilibrium strategy, by introducing a new and very interesting type of Riccati equation. Its main feature is that the generator depends on a delay term of the unknown.   In other words, this Riccati equation is a backward ordinary differential equation (ODE) with delay, which is equivalent to a forward ODE with advanced terms. Such an equation is essentially a Fredholm integral equation, whose solvability is challenging.
We overcome the difficulty by deriving a sharp a priori estimate and applying the Leray--Schauder fixed point theorem. To this end, we establish a comparison theorem between two matrix-valued nonlinear algebraic equations.
The convergence behavior of the solution to the Riccati equation as $T\to\infty$ is also provided.

\ms

\no{\bf Keywords.} stochastic linear-quadratic optimal control, sub-infinite horizon,  equilibrium strategy, time-inconsistency, Riccati equation.

\ms
\no{\bf AMS 2020 Mathematics Subject Classification.} \rm 93E20, 49N10, 49N35

\section{Introduction}\label{Sec:Intro}

Linear-quadratic (LQ) optimal control is a classical and fundamental problem in control theory.
The early studies focused on deterministic settings, where the state equation is a linear ordinary differential equation (ODE) and all functions involved are deterministic. Stochastic LQ problems were first addressed by Wonham \cite{Wonham1968} in 1968. Later, Bismut \cite{Bismut1976} provided a thorough analysis of stochastic LQ optimal control with random coefficients.
For the developments of stochastic LQ optimal control, we refer the reader to the books by Yong and Zhou \cite{Yong-Zhou1999}
and  Sun and Yong \cite{Sun-Yong2020}.

\ms

Based on the time horizon, stochastic LQ optimal control problems can be divided into finite horizon and infinite horizon cases.
Compared with the finite horizon case \cite{Yong-Zhou1999,Sun-Yong2020},
the optimal feedback strategy in the infinite horizon problem is time-invariant,
which is more practical in applications.
Moreover, since the infinite horizon integral in the cost functional is not always well-defined, it is closely related to the well-known concept of stabilizability. From the viewpoint of the Hamilton--Jacobi--Bellman (HJB) equation, the algebraic Riccati equation in infinite horizon problems corresponds to an elliptic equation, which is essentially different from the parabolic equation derived in finite horizon control problems. Therefore, the infinite horizon LQ optimal control problem is of independent interest.
Some recent developments in infinite horizon stochastic LQ optimal control problems can be found in, for example, Ait Rami and Zhou \cite{Ait-Rami-Zhou2000}, Yao, Zhang, and Zhou \cite{Yao-Zhou2001}, Huang, Li, and Yong \cite{Huang-2015}, Sun and Yong \cite{Sun-Yong-2018}, Wei and Yu \cite{Wei-Yu2021}, Li et al. \cite{Li-2022}, and L\"{u} \cite{LV2024}.

\ms

In this paper, we investigate an intermediate class of problems,
called \emph{sub-infinite horizon problems}, lying between finite horizon and infinite horizon problems.
Let $(\Om,\cF,\dbF,\dbP)$ be a complete filtered probability space, on which a one-dimensional standard Brownian motion $W$ is defined, whose natural filtration augmented by all the $\dbP$-null sets in $\cF$ is denoted by $\dbF\equiv\{\cF_s\}_{s\ges0}$.
For any given initial pair $(t,x)\in[0,\infty)\times\dbR^n$ and a running time horizon $T>0$,
consider the state equation over the finite horizon $[t,t+T]$:
\bel{state}\left\{\begin{aligned}
   dX(s) &= [AX(s) +Bu(s)]ds + [C X(s) +D u(s)]dW(s), \\
    X(t) &= x,
\end{aligned}\right.\ee
where  $A,\,C\in\dbR^{n\times n}$ and $B,\,D\in\dbR^{n\times m}$ are the given coefficients.
The process $u(\cd)$ is called the {\it control process}, which belongs to the space
\begin{align*}
\sU[t,t+T]&=\Big\{u:[t,t+T]\times\Om\to\dbR^m\bigm|u(\cd)\hb{~is $\dbF$-progressively measurable}, \\ &\qq~\dbE\int^{t+T}_t|u(s)|^2ds<\infty\Big\},
\end{align*}
and the unique  solution  of \rf{state} is called a {\it state process}.
To measure the performance of the control $u(\cd)$, we introduce the following cost functional:
\begin{align}
\cJ(t,x;u(\cd)) &= \dbE\int_t^{t+T}\[\lan QX(s),X(s)\ran+ \lan R u(s),u(s)\ran\]ds ,\label{cost}
\end{align}
with $Q\in \dbS^n$ and $R\in\dbS^m$ being given weighting matrices. The corresponding optimal control problem can be stated as follows.

\ms

\no{\bf Problem (LQ-Sub).} For any given initial pair $(t,x)\in [0,\infty)\times\dbR^n$, find a control $u^*(\cd)\in \sU[t,t+T]$ such that
\bel{inf-J}
\cJ(t,x;u^*(\cd))\les \cJ(t,x;u(\cd)),\qq\forall u(\cd)\in\sU[t,t+T].
\ee

\ms

The main feature of Problem (LQ-Sub) is that although the  running time horizon $T>0$ is finite,
the initial time $t$ can be chosen from the infinite time horizon $[0,\infty)$; for this reason, we call
Problem (LQ-Sub) a sub-infinite horizon stochastic LQ optimal control problem.
Besides mathematical considerations, there exist many motivations for sub-infinite horizon problems in practice.
For example, a company may exist for a very long time, so its initial time $t$ should be chosen from $[0,\infty)$,
while it mainly cares about its benefit over the next five years, for which the running time horizon $T$ should be finite.
Indeed, this type of control problem has been widely used to design feedback controls for ensuring system stability, as exemplified by Chen and Shaw \cite{ChenShaw1982}, Mayne and Michalska \cite{MayneMichalska1990}, Chen and Allg\"{o}wer \cite{Chen1998}, and the book by Kwon and Han \cite{KwonHan2005Book}.

\ms

For any fixed time horizon $[t,t+T]$, the optimal control of Problem (LQ-Sub) can be solved by a standard method:
Under the standard condition $Q\ges 0$ and $R>0$, the unique optimal control $u^*(\cd;t,x)$ is given by
the following feedback form:
\begin{align}
&u^*(s;t,x)=\Th^*(s;t)X^*(s;t,x)\nn\\
&:=-[R+D^\mathrm{T} P(s;t) D]^{-1}[B^\mathrm{T} P(s;t)+D^\mathrm{T} P(s;t)C]X^*(s;t,x),\q s\in[t,t+T],
\label{Intro-pre-solution}
\end{align}
where $X^*(\cd;t,x)$ is the unique solution of \rf{state} corresponding to the control $u^*(\cd)$,
and $P(\cd;t)$ is the unique solution of the following Riccati equation:
\bel{Intro-pre-riccati}\left\{
\begin{aligned}
&{dP(s;t)\over ds}+P(s;t)A+A^\mathrm{T} P(s;t)+C^\mathrm{T} P(s;t)C+Q-[P(s;t)B+C^\mathrm{T} P(s;t)D]\\
&\q\times[R+D^\mathrm{T} P(s;t)D]^{-1}[B^\mathrm{T} P(s;t)+D^\mathrm{T} P(s;t)C]=0,\qq s\in[t,t+T],\\
&P(t+T;t)=0.
\end{aligned}\right.
\ee
Through careful observation, one can find that for $t^\prime \in(t,t+T)$,
\bel{Intro-pre-riccati1}
P(s;t)\neq P(s;t^\prime),\qq s\in[t^\prime,t+T],
\ee
in general. Then, from \rf{Intro-pre-solution}, we have
\bel{Intro-pre-riccati2}
u^*(s;t,x)\neq u^*(s;t^\prime,X^*(t^\prime;t,x)) ,\qq s\in[t^\prime,t+T],
\ee
which means that the  optimal control selected at a given initial pair will not stay optimal thereafter.
In other words, Problem (LQ-Sub) is \emph{time-inconsistent}.

\ms

Given the time-inconsistent nature of Problem (LQ-Sub) described above,
we will study its equilibrium strategy (see Definition \ref{Def:ES}) from a dynamic game viewpoint.
The earliest mathematical study of time-inconsistent optimal control is due to Strotz \cite{Strotz-1955}, followed by
Ekeland and Lazrak \cite{Ekeland-Lazrak-2010}, Bj\"ork, Khapko, and Murgoci \cite{Bjork-2017},
and Yong \cite{Yong2012,Yong2014}, to mention a few. In the linear-quadratic framework,
Hu, Jin, and Zhou \cite{Hu-2012,Hu-2017} studied the open-loop equilibria using forward-backward stochastic differential equations;
Yong \cite{Yong2017} constructed a closed-loop equilibrium strategy by the multi-person differential game method;
Dou and L\"{u} \cite{Dou-2020} extended the results of \cite{Yong2017} to the infinite-dimensional system;
Wang \cite{Wang2020} provided a characterization of closed-loop equilibrium strategies for  mean-field problems;
Cai et al. \cite{Cai-2022} studied  the problem of controlled ODE systems;
Li et al. \cite{Li-2026} studied the problem with an indefinite cost functional;
L\"{u} and Ma \cite{Lu-2024} and L\"{u},  Ma, and Wang \cite{Lu-Ma-Wang-2025}  studied the forward-backward problem
using an a priori estimate method.
However, all the above works focus on finite horizon problems, and none of them can address Problem (LQ-Sub).
This paper aims to resolve this issue, and some interesting new phenomena arise.

\ms

The main results of this paper can be briefly summarized as follows.

\ms

$\bullet$ The following new type of Riccati equation is introduced:
 \bel{Intr-ERiccati}\left\{
\begin{aligned}
&\dot{\Si}(s)+\Si(s)[A+B\Psi]+[A+B\Psi]^\mathrm{T} \Si(s)+[C+D\Psi]^\mathrm{T} \Si(s)[C+D\Psi]\\
&\q+[Q+\Psi^\mathrm{T} R\Psi]=0,\qq 0\les  s \les T,\\
&\Si(T)=0,
\end{aligned}\right.
\ee
with
\bel{Intr-ERiccati-Th}
\Psi=-[R+D^\mathrm{T} \Si(0)D]^{-1}[B^\mathrm{T} \Si(0)+D^\mathrm{T} \Si(0)C].
\ee
We show that if the Riccati equation \rf{Intr-ERiccati}--\rf{Intr-ERiccati-Th} admits a positive semidefinite solution $\Si(\cd)$,
then $\bar\Th_{ES}:=\Psi$ is an equilibrium strategy of Problem (LQ-Sub) (see Theorem \ref{theorem:verification}).
It is worth pointing out  that while the pre-committed optimal strategy $\Th^*(\cd)$ given by \rf{Intro-pre-solution} is time-varying,
the equilibrium strategy $\bar\Th_{ES}$ is time-invariant.
Thus, the equilibrium strategy $\bar\Th_{ES}$ retains the important \emph{time-invariance property}
of optimal feedback strategies in infinite horizon control problems.

\ms
$\bullet$
Under the standard condition \ref{ass:H1} and the stabilizability condition \ref{ass:H2},
the solvability of \emph{delayed Riccati equation} \rf{Intr-ERiccati}--\rf{Intr-ERiccati-Th} is established
(see Theorem \ref{theorem:solvability-SERE}).
Moreover, we show that the uniqueness of  solutions to \rf{Intr-ERiccati}--\rf{Intr-ERiccati-Th} does not hold in general,
and that the stabilizability condition is also necessary for its solvability. Two counterexamples (see Examples \ref{example:non-unique} and \ref{example:no_solution}) are provided to illustrate these claims.

\ms

Recall from \cite{Yong-Zhou1999,Sun-Yong2020} that the classical Riccati equation associated with stochastic LQ optimal control problems is given by
\bel{Intro-classical-riccati}\left\{
\begin{aligned}
&\dot{P}(s)+P(s)A+A^\mathrm{T} P(s)+C^\mathrm{T} P(s)C+Q-[P(s)B+C^\mathrm{T} P(s)D]\\
&\q\times[R+D^\mathrm{T} P(s)D]^{-1}[B^\mathrm{T} P(s)+D^\mathrm{T} P(s)C]=0,\qq s\in[0,T],\\
&P(T)=0.
\end{aligned}\right.
\ee
We see that in \rf{Intro-classical-riccati}, the derivative $\dot{P}(s)$ at time $s$ depends only on $P(s)$, whereas in \rf{Intr-ERiccati}--\rf{Intr-ERiccati-Th}, the derivative $\dot{\Si}(s)$ at time $s$ depends, through $\Psi$, on the historical value $\Si(0)$. Thus, there is an essential difference between \rf{Intro-classical-riccati} and \rf{Intr-ERiccati}--\rf{Intr-ERiccati-Th}. To emphasize the new feature of \rf{Intr-ERiccati}--\rf{Intr-ERiccati-Th}, we refer to it as a \emph{delayed Riccati equation}.

\ms
Note that \rf{Intr-ERiccati} is a backward equation with a terminal condition.
By time reversal, the delayed Riccati equation \eqref{Intr-ERiccati}--\eqref{Intr-ERiccati-Th} is equivalent to the following forward equation with an advanced term $\G(T)$:

\bel{Intr-ERiccati-Gamma}\left\{
\begin{aligned}
&\dot{\G}(s)-\G(s)[A+B\Psi]-[A+B\Psi]^\mathrm{T} \G(s)\\
&\q-[C+D\Psi]^\mathrm{T} \G(s)[C+D\Psi]-[Q+\Psi^\mathrm{T} R\Psi]=0,\qq 0\les  s \les T,\\
&\G(0)=0,
\end{aligned}\right.
\ee
with
\bel{Intr-ERiccati-Th-Gamma}
\Psi=-[R+D^\mathrm{T} \G(T)D]^{-1}[B^\mathrm{T} \G(T)+D^\mathrm{T} \G(T)C].
\ee
To the best of our knowledge, this type of Riccati equation is derived for the first time in control theory.

\ms

Due to the presence of $\Si(0)$ in \rf{Intr-ERiccati-Th}, the delayed Riccati equation \rf{Intr-ERiccati}--\rf{Intr-ERiccati-Th} is essentially of Fredholm type (see Example \ref{example1}). Consequently, the standard Banach fixed point theorem is not applicable, and its solvability becomes far more challenging than that of the classical Riccati equation \rf{Intr-ERiccati-Gamma} and the Volterra-type equilibrium Riccati equation derived by \cite{Yong2017}.
We overcome the difficulty by deriving a sharp a priori estimate (see Proposition \ref{bound-Gamma})
and applying the Leray--Schauder fixed point theorem  (see Lemma \ref{lem:leray-schauder}).
A key step is that we establish a comparison theorem between two matrix-valued nonlinear algebraic equations
(see Lemma \ref{comparison}).
Some illustrative examples are also given (See Examples \ref{example:non-unique} and \ref{example:no_solution}).

\ms

$\bullet$
Let $\Si_T(\cd)$ be a positive semidefinite solution to the delayed Riccati equation \rf{Intr-ERiccati}--\rf{Intr-ERiccati-Th},
where the subscript  $T$ indicates that the equation is defined over the interval $[0,T]$.
We show that
\bel{Intr-convergence}
\lim_{T\to\infty}\Si_T(0)=P_\infty,
\ee
where $P_\infty$ is the unique stabilizing solution  to the algebraic Riccati equation associated with the infinite horizon stochastic LQ problem corresponding to \rf{state}--\rf{cost} (see Theorem \ref{thm:limit}).
This means that, as $T\to\infty$, the equilibrium strategy $\bar\Th_{ES}$ converges to
the optimal strategy of the infinite horizon problem, and that the  a priori estimate in Proposition \ref{bound-Gamma} is sharp.

\ms

The rest of this paper is organized as follows.
In Section \ref{sec:pre}, we collect some preliminary results and introduce several notions for Problem (LQ-Sub). In Section \ref{sec:VT}, the equilibrium strategy is introduced, the delayed Riccati equation is derived, and the verification theorem is proved. The solvability of the delayed Riccati equation is established in Section \ref{sec:solv}, and its convergence behavior is presented in Section \ref{sec:con}. We conclude the paper in Section \ref{sec:conclu}. Some lengthy proofs are given in the Appendix.

\section{Preliminary}\label{sec:pre}
Throughout this paper, we adopt the following notation. For a matrix $M$, $M^\mathrm{T}$ denotes its transpose, $\tr(M)$ its trace, and $|M|$ its Frobenius norm. The space $\dbR^{n\times m}$ consists of all $n\times m$ real matrices, equipped with the Frobenius inner product $\lan M,N\ran = \tr(M^\mathrm{T} N)$.  Let $\dbS^n$ be the subspace of symmetric matrices in $\dbR^{n\times n}$, and let $\dbS_+^n$ be the subset of positive semidefinite matrices in $\dbS^n$. For $M,N\in\dbS^n$, we write $M \ges N$ (resp. $M > N$) to mean that $M-N$ is positive semidefinite (resp. positive definite). Given a subinterval $[a,b]\subseteq [0,\infty)$ (where $b$ may be $\infty$) and a Euclidean space $\mathbb{H}$ (which may be $\dbR^n$, $\dbR^{n\times m}$, $\dbS^n$, etc.), we introduce the following spaces of functions and processes.
\begin{align*}
&C(a,b;\dbH)=\Big\{\f:[a,b]\to\dbH~\big|~\f(\cd)
   \hb{~is continuous, $\sup_{s\in[a,b]}|\f(s)|<\infty$}\Big\}. \\
&L_\dbF^2(a,b;\dbH)=\Big\{\f:[a,b]\times\Om\to\dbH~\big|~\f(\cd)
   \hb{~is $\dbF$-progressively measurable,}\\
  &\qq\qq\qq\qq  \dbE\int_a^b|\f(s)|^2ds<\infty\Big\}. \\
&L_\dbF^2(\Om;C([a,b];\dbH))=\Big\{\f:[a,b]\times\Om\to\dbH~\big|~\f(\cd)
   \hb{~is $\dbF$-adapted and continuous,}\\
&\qq\qq\qq\qq\qq\qq \dbE\big[\sup_{s\in[a,b]}|\f(s)|^2\big]<\infty\Big\}.
\end{align*}

\ms

Recall that $A,\,C\in\dbR^{n\times n}$ and $B,\,D\in\dbR^{n\times m}$.
We assume that the weighting matrices satisfy the following \emph{standard condition}.
\begin{taggedassumption}{(H1)}\label{ass:H1}
The weighting matrices of the cost functional \rf{cost} satisfy
$$
Q> 0, \q R>0.
$$
\end{taggedassumption}

The infinite horizon problem associated with Problem (LQ-Sub) is formulated as follows:
Consider the state equation \rf{state} over $[t,\infty)$ and  the following cost functional:
\begin{align}
\cJ_\infty (t,x;u(\cd)) &= \dbE\int_t^\infty\[\lan QX(s),X(s)\ran+ \lan R u(s),u(s)\ran\]ds .\label{cost-infty}
\end{align}

Note that even if we take $u(\cdot)\in L^2_{\dbF}(t,\infty;\dbR^m)$,
it is still not guaranteed that the unique solution $X(\cdot;t,x,u(\cdot))$ of \eqref{state}
belongs to $L^2_{\dbF}(t,\infty;\dbR^n)$; consequently, \eqref{cost-infty} may not be well-defined.
Thus, the admissible control set is defined by
$$
\sU_{ad}(t,x)=\Big\{ u(\cd)\in L^2_{\dbF}(0,\infty;\dbR^m) ~\Big|~\dbE\int_t^\infty |X(s;t,x,u(\cd))|^2ds<\infty\Big\},\q (t,x)\in[0,\infty)\times\dbR^n.
$$

\ms

To ensure that $\sU_{ad}(t,x)$ is nonempty, we introduce the  $L^2$-stabilizability condition for the state equation \eqref{state}.

\begin{definition}\label{def:stablizable}
The system $[A;C]$ is said to be \emph{$L^2$-stable} if
for any $x\in\dbR^n$, the unique solution $X(\cdot)=X(\cdot;x)$ of the following  system:
\bel{def:stablizable1}
\left\{\begin{aligned}
   dX(s) &= AX(s) ds + C X(s) dW(s), \qquad s\in[0,\infty),\\
    X(0) &= x,
\end{aligned}\right.
\ee
belongs to $L^2_{\dbF}(0,\infty;\dbR^n)$; that is,
\bel{def:stablizable2}
\dbE\int_0^\infty |X(s)|^2\,ds < \infty.
\ee
The system $[A,B;C,D]$ is said to be \emph{$L^2$-stabilizable} if there exists a matrix $\Theta\in\dbR^{m\times n}$ such that
the system $[A+B\Th;C+D\Th]$ is $L^2$-stable.
In this case, $\Th$ is called a \emph{stabilizer} of the system $[A,B;C,D]$.
\end{definition}

\begin{remark}
From \cite{Huang-2015,Sun-Yong-2018},
 the system $[A;C]$ is $L^2$-stable if and only if there exists a matrix $P\in\dbS^n_+$ such that
$$
PA+A^\mathrm{T} P+C^\mathrm{T} PC<0.
$$
A typical example is the case where $[A;C]$ satisfies  $A+A^\mathrm{T}+C^\mathrm{T} C<0$.
Moreover, it is clear that the system $[A,B;C,D]$ is $L^2$-stabilizable if $[A;C]$ is $L^2$-stable.
\end{remark}

\begin{taggedassumption}{(H2)}\label{ass:H2}
The system $[A,B;C,D]$ is $L^2$-stabilizable.
\end{taggedassumption}

Under \ref{ass:H2}, by Sun and Yong \cite{Sun-Yong-2018}, we have $\sU_{ad}(t,x)\neq\emptyset$ for any $(t,x)\in[0,\infty)\times\dbR^n$.
Further, under \ref{ass:H1} and \ref{ass:H2},
the infinite horizon stochastic LQ problem with state equation \rf{state} and cost functional \rf{cost-infty}
can be explicitly solved in the following manner.

\ms

The algebraic Riccati equation associated with the infinite horizon problem \rf{state} and \rf{cost-infty} is given by
\begin{align}
&A^\mathrm{T} P_\infty + P_\infty A + C^\mathrm{T} P_\infty C + Q- (B^\mathrm{T} P_\infty + D^\mathrm{T} P_\infty C)^\mathrm{T} \nn\\
&\qq \times(R + D^\mathrm{T} P_\infty D)^{-1} (B^\mathrm{T} P_\infty + D^\mathrm{T} P_\infty C) = 0. \label{ARE-infty}
\end{align}

\begin{definition}
We call $P_\infty\in\dbS^n$ a \emph{static stabilizing solution} to \rf{ARE-infty} if  it satisfies \rf{ARE-infty},
and the feedback strategy $\Th_\infty$, given by
\bel{def-SS-ARE}
\Th_\infty=-(R + D^\mathrm{T} P_\infty D)^{-1} (B^\mathrm{T} P_\infty + D^\mathrm{T} P_\infty C),
\ee
is a stabilizer of the system $[A,B;C,D]$.
\end{definition}

\begin{proposition}\label{prop:infinite}
Let \ref{ass:H1} and \ref{ass:H2} hold. Then the algebraic Riccati equation \rf{ARE-infty} admits a unique static stabilizing solution
$P_\infty\in\dbS^n_+$. The infinite horizon problem with state equation \rf{state} and cost functional \rf{cost-infty}
admits a unique optimal control $u^*(\cd)\in\sU_{ad}(t,x)$, which admits the closed-loop representation:
\bel{prop:infinite1}
u^*(s)=\Th_\infty X^*(s),\qq s\in[t,\infty),
\ee
where $X^*(\cd)$ is the unique solution to the closed-loop system:
\bel{prop:infinite2}
\left\{\begin{aligned}
   dX^*(s) &= [A +B\Theta_\infty ]X^*(s)ds + [C  +D\Theta_\infty] X^*(s)dW(s), \qquad s\in[t,\infty),\\
    X^*(t) &= x.
\end{aligned}\right.
\ee
Moreover, the value function is given by
\bel{prop:infinite3}
\lan P_\infty x,x\ran=\cJ_\infty(t,x;u^*(\cd))=\inf_{u(\cd)\in\sU_{ad}(t,x)}\cJ_\infty(t,x;u(\cd)).
\ee
\end{proposition}

\ms

The following is the well-known \emph{Leray--Schauder fixed point theorem}  (see \cite[Theorem 9.12-3]{Ciarlet2024}),
which will play an important role in establishing the solvability of the delayed Riccati equation derived in this paper.

\begin{lemma}\label{lem:leray-schauder}
Let $X$ be a Banach space and let $f : X \times [0,1] \to X$ be a compact mapping with the following properties:
$$
f(x,0) = 0 \quad \text{for all } x \in X,
$$
and there exists $r>0$ such that
$$
\{ x \in X;\ f(x,\sigma) = x \text{ for some } 0 \les \sigma \les 1 \} \subset B(0;r).
$$
Then the mapping $f(\cdot,1) : X \to X$ has at least one fixed point in the closed ball $\overline{B(0;r)}$.
\end{lemma}

\section{Equilibrium strategy and verification theorem}\label{sec:VT}

Motivated by \cite{Ekeland-Lazrak-2010,Hu-2012,Bjork-2017,Yong2017}, we introduce the following definition of equilibrium strategy of Problem (LQ-Sub).

\begin{definition}\label{Def:ES}
We call $\bar\Th_{ES}$  an \emph{equilibrium strategy}, if
\bel{def-subequilibrium1}
\liminf_{\e\to 0^+} {\cJ(t,\bar X(t);u^\e(\cd) )-\cJ(t,\bar X(t);\bar\Th_{ES}\bar X(\cd) )\over\e}\ges 0,\ee
for any  $t\in[0,\infty)$ and $u(\cd)\in L_{\dbF}^2(0,\infty;\dbR^m)$, where $\bar X(\cd)$ is the unique solution to the following
closed-loop system:
\bel{def-subequilibrium3}
\left\{\begin{aligned}
   d\bar X(s) &= [A +B\bar\Theta_{ES}]\bar X(s)ds + [C+D\bar\Theta_{ES}]\bar X(s)dW(s), \qquad s\in[0,\infty),\\
    \bar X(0) &= x,
\end{aligned}\right.
\ee
and $u^\e(\cd)$ is the perturbed control process defined by
\bel{def-subequilibrium2}
u^\e(s)\deq\left\{\2n\ba{ll}
\ds\bar\Th_{ES}X^\e(s),\q&s\in[t+\e,t+T];\\
\ns\ds u(s),\q&s\in[t,t+\e),\ea\right.\ee
with $X^\e(\cd)\deq  X(\cd;t,\bar X(t),u^\e(\cd))$ being uniquely determined by
\bel{def-subequilibrium4}
\left\{\begin{aligned}
   dX^\e (s) &= [AX^\e (s) +Bu^\e(s)]ds + [C X^\e(s) +Du^\e(s)]dW(s), \qquad s\in[t,t+T],\\
    X^\e(t) &= \bar X(t).
\end{aligned}\right.
\ee

\end{definition}

\ms

\begin{remark}
Regarding Definition \ref{Def:ES}, three  points deserve emphasis.

\ms
\noindent
$\bullet$
The intuition behind this definition is  similar to that in \cite{Ekeland-Lazrak-2010,Hu-2012,Bjork-2017,Yong2017}.
It means that, at each time $t$, the controller is effectively engaged in a game against all her future selves. Once she follows the strategy $\bar\Th_{ES}$, she finds no reason to change it at any $t\in[0,\infty)$, because a unilateral deviation at any single instant would not yield a better outcome.

\ms
\noindent
$\bullet$
The equilibrium strategy is required to be \emph{time-invariant}.
This is because the time-invariance property of the optimal strategy is arguably the most important feature in infinite horizon
stochastic LQ problems, and we expect the equilibrium strategy to preserve this property.
This requirement introduces an essential difficulty.

\ms
\noindent
$\bullet$
Note that for any given initial time $t$, the current cost functional \rf{cost} is defined only over the finite horizon $[t,t+T]$. Hence, the equilibrium strategy $\bar\Th_{ES}$ is not necessarily a stabilizer of the system $[A,B;C,D]$.
\end{remark}

\subsection{Verification theorem}
Note that we can rewrite the cost functional \rf{cost} as the following form:
\begin{align}
\cJ(t,x;u(\cd)) &= \dbE\int_t^\infty \l(s-t) \[\lan QX(s),X(s)\ran+ \lan R u(s),u(s)\ran\]ds ,\label{cost-RW}
\end{align}
where
\bel{def-lambda}
\l(s)=\left\{
\begin{aligned}
&1,\qq s\in[0,T];\\
&0,\qq s\in(T,\infty).
\end{aligned}
   \right.
\ee
Then, motivated by the form of the equilibrium Riccati equation derived by Yong \cite{Yong2017} for the finite horizon problem, we introduce the following infinite horizon equilibrium Riccati equation:
\begin{align}
&{d\Si(s;t)\over ds}+\Si(s;t)[A+B\Psi(s)]+[A+B\Psi(s)]^\mathrm{T} \Si(s;t)+[C+D\Psi(s)]^\mathrm{T} \Si(s;t)\nn\\
&\q\times[C+D\Psi(s)]+\l(s-t)[Q+\Psi(s)^\mathrm{T} R\Psi(s)]=0,\qq 0\les t\les s<\infty,\label{P-Riccati-1}
\end{align}
where
\bel{P-strategy-1}
\Psi(s)=-[R+D^\mathrm{T} \Si(s;s)D]^{-1}[B^\mathrm{T} \Si(s;s)+D^\mathrm{T} \Si(s;s)C],\qq s\in[0,\infty).
\ee
Next, substituting $\l(\cd)$ into the above, we have
\bel{P-Riccati-2}
\left\{
\begin{aligned}
&{d\Si(s;t)\over ds}+\Si(s;t)[A+B\Psi(s)]+[A+B\Psi(s)]^\mathrm{T} \Si(s;t)+[C+D\Psi(s)]^\mathrm{T}\Si(s;t)\\
&\q\times[C+D\Psi(s)]+[Q+\Psi(s)^\mathrm{T} R\Psi(s)]=0,\qq 0\les t\les s<t+T<\infty,\\
&\Si(t+T;t)=0,\qq t\in[0,\infty).
\end{aligned}\right.
\ee
Note that the equilibrium strategy $\bar\Th_{ES}$ is required to be time-invariant.
Then $\Psi(\cd)$ should be independent of the time variable.
Thus, we should find a solution $\Si(\cd;\cd)$ to \rf{P-Riccati-2} with the additional property
that $\Si(s;s)\equiv \Si$ for some matrix $\Si\in\dbS^n_+$.
With this property, we have the following equivalent representation of \rf{P-Riccati-2}:
\bel{}
\Si(s;t)=\Si(s-t),
\ee
with $\Si(\cd)$ satisfying
 \bel{ERiccati}\left\{
\begin{aligned}
&\dot{\Si}(s)+\Si(s)[A+B\Psi]+[A+B\Psi]^\mathrm{T} \Si(s)+[C+D\Psi]^\mathrm{T} \Si(s)[C+D\Psi]\\
&\q+[Q+\Psi^\mathrm{T} R\Psi]=0,\qq 0\les  s \les T,\\
&\Si(T)=0,
\end{aligned}\right.
\ee
where
\bel{ERiccati-Psi}
\Psi=-[R+D^\mathrm{T} \Si(0)D]^{-1}[B^\mathrm{T} \Si(0)+D^\mathrm{T} \Si(0)C].
\ee

\ms
It is particularly remarkable that  $\dot{\Si}(s)$ depends on the delayed term $\Si(0)$ through $\Psi$, and
thus we call \rf{ERiccati}--\rf{ERiccati-Psi}  a \emph{delayed Riccati equation}.
By  time reversal, \rf{ERiccati} is equivalent to a
forward ODE with advanced arguments. To our best knowledge, it is the first time that
such type of Riccati equations is derived.

\ms

With the delayed Riccati equation \rf{ERiccati}--\rf{ERiccati-Psi} derived in the above,
the equilibrium strategy can be constructed in the following sense.

\begin{theorem}\label{theorem:verification}
Suppose that the delayed Riccati equation \rf{ERiccati}--\rf{ERiccati-Psi} admits a solution
$\Si(\cd)\in C([0,T];\dbS^n_+)$.
Then the linear feedback strategy $\bar\Theta_{ES}:=\Psi$, with $\Psi$ given by \rf{ERiccati-Psi},
is an equilibrium strategy.
\end{theorem}

\begin{proof}
 For any fixed $t\in[0,\infty)$ and $u(\cd)\in L_{\dbF}^2(t,\infty;\dbR^m)$,
define the control process $u^\e(\cd\,,\cd)$ by \rf{def-subequilibrium2} with $0<\e<T$.
With the initial pair $(t,\bar X(t))$,
taking the control $u^\e(\cd\,,\cd)$, the corresponding state equation \rf{state} and cost functional \rf{cost} become
\bel{theorem:verification-p1}\left\{\begin{aligned}
   dX^\e(s) &= [AX^\e(s) +Bu(s)]ds + [C X^\e(s) +D u(s)]dW(s),\q s\in[t,t+\e); \\
   dX^\e(s) &= [A+B\bar\Th_{ES}] X^\e(s)ds +[C+D\bar\Th_{ES}]  X^\e(s)dW(s),\q s\in[t+\e,t+T], \\
    X^\e(t) &= \bar X(t),
\end{aligned}\right.\ee
and
\begin{align}
\cJ(t,\bar X(t);u^\e(\cd)) &= \dbE\Big[\int_t^{t+\e}\(\lan QX^\e(s),X^\e(s)\ran+ \lan Ru(s),u(s) \ran\)ds\nn\\
&\q+\int_{t+\e}^{t+T}\(\lan QX^\e(s),X^\e(s)\ran+ \lan R\bar\Th_{ES}X^\e(s),\bar\Th_{ES}X^\e(s) \ran\)ds\Big],\label{theorem:verification-p2}
\end{align}
respectively.
Applying  It\^{o}'s formula to the mapping $s\mapsto\lan \Si (s-t)X^\e(s),X^\e(s)\ran$ over $[t+\e,t+T]$, we have
\begin{align*}
&d\lan \Si (s-t)X^\e(s),X^\e(s)\ran=\[\lan \dot{\Si} (s-t)X^\e(s),X^\e(s)\ran\\
&\q+\lan \Si (s-t)[A+B\bar\Th_{ES}]X^\e(s),X^\e(s)\ran+\lan \Si (s-t)X^\e(s),[A+B\bar\Th_{ES}]X^\e(s)\ran\\
&\q+\lan \Si (s-t)[C+D\bar\Th_{ES}]X^\e(s),[C+D\bar\Th_{ES}]X^\e(s)\ran\]ds\\
&\q+\[\lan \Si (s-t)[C+D\bar\Th_{ES}]X^\e(s),X^\e(s)\ran+\lan \Si (s-t)X^\e(s),[C+D\bar\Th_{ES}]X^\e(s)\ran\] dW(s)\\
&=-\[\lan [Q+\bar\Th^\mathrm{T}_{ES}R \bar\Th_{ES}]X^\e(s),X^\e(s)\ran\]ds+\[\lan \Si (s-t)[C+D\bar\Th_{ES}]X^\e(s),X^\e(s)\ran\\
&\q+\lan \Si (s-t)X^\e(s),[C+D\bar\Th_{ES}]X^\e(s)\ran\] dW(s).
\end{align*}
Then,
\bel{theorem:verification-p7}
\dbE\int_{t+\e}^{t+T}\[\lan [Q+\bar\Th^\mathrm{T}_{ES}R \bar\Th_{ES}]X^\e(s),X^\e(s)\ran\]ds=\dbE\big[\lan\Si(\e) X^\e(t+\e),X^\e(t+\e)\ran\big].
\ee
Substituting the above into \rf{theorem:verification-p2} yields that
\begin{align}
\cJ(t,\bar X(t);u^\e(\cd)) &=\dbE\Big[\int_t^{t+\e}\(\lan QX^\e(s),X^\e(s)\ran+ \lan Ru(s),u(s) \ran\)ds\nn\\
&\qq+\lan\Si(\e) X^\e(t+\e),X^\e(t+\e)\ran\].\label{theorem:verification-p3}
\end{align}
Note that the problem with state equation \rf{theorem:verification-p1} and cost functional \rf{theorem:verification-p3}
is a classical finite horizon LQ optimal control problem.
By the standard results of stochastic LQ optimal problems (see Yong and Zhou \cite{Yong2012},
for example), we have
\bel{theorem:verification-p10}
\cJ(t,\bar X(t);u^\e(\cd)) \ges\dbE\big[ \lan P_\e(t)\bar X(t),\bar X(t)\ran\big],
\ee
where
\bel{theorem:verification-p4}\left\{
\begin{aligned}
&\dot{P}_\e(s)+P_\e(s)A+A^\mathrm{T} P_\e(s)+C^\mathrm{T} P_\e(s)C+Q- [P_\e(s)B+C^\mathrm{T} P_\e(s)D]\\
&\q\times[R+D^\mathrm{T} P_\e(s)D]^{-1}[B^\mathrm{T} P_\e(s)+D^\mathrm{T} P_\e(s)C ]=0 ,\qq  t\les s\les t+\e,\\
&P_\e(t+\e)=\Si(\e).
\end{aligned}\right.
\ee
On the other hand,  applying  It\^{o}'s formula to the mapping $s\mapsto\lan \Si (s-t)\bar X(s),\bar X(s)\ran$ over $[t,t+T]$,
in a similar manner to \rf{theorem:verification-p7}, we have
\bel{theorem:verification-p9}
\cJ(t,\bar X(t);\bar\Th_{ES}\bar X(\cd) )=\dbE\big[\lan\Si(0)\bar X(t),\bar X(t)\ran\big].
\ee
Then from \rf{theorem:verification-p10} and \rf{theorem:verification-p9}, we have
\begin{align}
&\cJ(t,\bar X(t);u^\e(\cd)) -\cJ(t,\bar X(t);\bar\Th_{ES}\bar X(\cd) )\nn\\
&\ges \dbE\big[ \lan P_\e(t)\bar X(t),\bar X(t)\ran-\lan\Si(0)\bar X(t),\bar X(t)\ran\big].\label{theorem:verification-p12}
\end{align}
Note that $\Si(\cd)$ is bounded, and then $P_\e(\cd)$ is uniformly bounded.
Hence both $\Si(\cd)$ and $P_\e(\cd)$ are uniformly Lipschitz continuous.
Thus, there exists a constant $K>0$, which is independent of $\e$, such that
\begin{align}
|P_\e(s)-\Si(s-t)|&\les |P_\e(s)-P_\e(t+\e)|+|P_\e(t+\e)-\Si(s-t)|\nn\\
&= |P_\e(s)-P_\e(t+\e)|+|\Si(\e)-\Si(s-t)|\nn\\
&\les K\e,\qq \forall s\in[t,t+\e]. \label{theorem:verification-p13}
\end{align}
Note that $\Si(\cd-t)$ satisfies the following equation:
\bel{theorem:verification-p5}
\left\{
\begin{aligned}
&\dot{\Si}(s-t)+\Si(s-t)[A+B\Psi]+[A+B\Psi]^\mathrm{T} \Si(s-t)\\
&\q+[C+D\Psi]^\mathrm{T} \Si(s-t)[C+D\Psi]+[Q+\Psi^\mathrm{T} R\Psi]=0,\qq t\les  s \les t+\e,\\
&\Si(t+\e-t)=\Si(\e),
\end{aligned}\right.
\ee
with $\Psi$ given by \rf{ERiccati-Psi}.
Comparing \rf{theorem:verification-p4} and \rf{theorem:verification-p5},
by \rf{theorem:verification-p13} and the uniform boundedness of $P_\e(\cd)$ and $\Si(\cd)$, we get
$$
|P_\e(t)-\Si(0 )|\les K\int_t^{t+\e} \big[|P_\e(s)-\Si(s-t)|+|\Si(s-t)-\Si(0)|\big]ds\les K\e^2.
$$
Combining the above with \rf{theorem:verification-p12}, we get
\begin{align*}
&\liminf_{\e\to 0^+}{\cJ(t,\bar X(t);u^\e(\cd)) -\cJ(t,\bar X(t);\bar\Th_{ES}\bar X(\cd) )\over \e}\\
&\ges\liminf_{\e\to 0^+}{ \dbE\big[ \lan P_\e(t)\bar X(t),\bar X(t)\ran-\lan\Si(0)\bar X(t),\bar X(t)\ran\big]\over \e}\\
&=0,
\end{align*}
which implies the local optimality of the equilibrium strategy $\bar\Th_{ES}$.
\end{proof}

\section{Solvability of delayed Riccati equation}\label{sec:solv}

Since the dependence of $\Si(0)$, the solvability of   \rf{ERiccati}--\rf{ERiccati-Psi} is not easy.
%
%
%
%
%
%
Before going further, let us  present the following simple and illustrative example.

\begin{example}\label{example1}
Consider the following linear backward delayed equation:
\bel{example-1}
\dot{\Si}(s)=-\Si(s)+{1\over 1-e}\Si(0),\qq s\in[0,1];\qq \Si(1)=1.
\ee
Clearly, if \rf{example-1} has a solution, then we have
$$
\Si(s)=e^{1-s}-\int_s^1 e^{r-s} dr {1\over 1-e}\Si(0),\qq s\in[0,1].
$$
Thus,
$$
\Si(0)=e-\int_0^1 e^r dr {1\over 1-e}\Si(0)=e+\Si(0),
$$
which yields a contradiction.
\end{example}

We see that \rf{example-1} is a very simple linear equation, yet it does not have a solution.
Since the derivative term $\dot{\Si}(s)$ depends on $\Si(0)$, it is influenced by the entire trajectory  of $\Si(\cd)$.
Hence the delayed backward equation (i.e., \rf{ERiccati}--\rf{ERiccati-Psi} or \rf{example-1})
is essentially a {\it Fredholm integral equation}. Thus, the solvability of \rf{ERiccati}--\rf{ERiccati-Psi} becomes much more difficult than that of classical Riccati equations, and the standard Banach fixed point theorem cannot be applied. We overcome the difficulty by introducing a new method.

\ms

The main result of this section is given as follows.

\begin{theorem}\label{theorem:solvability-SERE}
Let \ref{ass:H1} and \ref{ass:H2} hold.
Then the delayed Riccati equation \rf{ERiccati}--\rf{ERiccati-Psi} admits a solution $\Si(\cd)\in C([0,T];\dbS^n_+)$.
\end{theorem}

\begin{remark}
Under \ref{ass:H1} and \ref{ass:H2}, the uniqueness of solutions to the delayed Riccati equation \rf{ERiccati}--\rf{ERiccati-Psi}
still does not hold. The following is a simple example.
\end{remark}

\begin{example}\label{example:non-unique}
Consider the delayed Riccati  equation \rf{ERiccati}--\rf{ERiccati-Psi} with the parameter setting:
$$
A = 1,\quad B = 1,\quad C =0, \quad D = 0, \quad Q = 0.01,\quad R = 1,\quad T = 1.
$$
One can easily verify that Assumptions \ref{ass:H1} and \ref{ass:H2} are satisfied.
Then, the equation \rf{ERiccati}--\rf{ERiccati-Psi} can be rewritten as follows:
 \bel{ERiccati-example1}\left\{
\begin{aligned}
&\dot{\Si}(s)+2\Si(s)[1 - \Si(0)]+[0.01+ \Si(0)^2]=0,\qq 0\les  s \les 1,\\
&\Si(1)=0,
\end{aligned}\right.
\ee
Direct integration of the above differential equation shows that the initial value $\Si(0)$ satisfies the following scalar algebraic equation:
$$
f(\Si(0)) = 0,
$$
where $f(\cd)$ is defined as
$$
f(p) = \int_0^1 e^{2(1-p)r} \left(0.01 + p^2\right) dr - p =
\begin{cases}
\displaystyle \frac{0.01 + p^2}{2(1-p)} \left(e^{2(1-p)} - 1\right) - p, & p \neq 1,\\[1em]
0.01, & p = 1.
\end{cases}
$$
Evaluating the function at selected points yields
$$
f(0) \approx 0.0319 > 0,\quad
f(0.1) \approx -0.0439 < 0,\quad
f(1) = 0.01 > 0,\quad
f(2) \approx -0.2663 < 0.
$$
By the intermediate value Theorem, $f(p) = 0$ has one root in each of the intervals $(0,0.1)$, $(0.1,1)$, and $(1,2)$.
Each root corresponds to a distinct initial value $\Si(0)$ and thereby generates a unique solution $\Si(\cdot)$ to the original Riccati equation, which is given by
$$
\Si(s)=\int_s^1 e^{2(1-\Si(0))(r-s)} \left(0.01 + \Si(0)^2\right) dr,\qq s\in[0,1].
$$
 This means the considered equation has at least three positive solutions.
\end{example}

\begin{remark}
Without Assumption \ref{ass:H2}, the solvability of the delayed Riccati equation \rf{ERiccati}--\rf{ERiccati-Psi} may fail to hold. This indicates that the $L^2$-stabilizability condition is necessary for its solvability.
We give an example  below.
\end{remark}

\begin{example}\label{example:no_solution}
Consider the  delayed Riccati equation \rf{ERiccati}--\rf{ERiccati-Psi} with the parameter setting:
\begin{align}
& A = A_c = \begin{pmatrix}
1 & 1 \\[2pt]
0 & c
\end{pmatrix},\quad
B = \begin{pmatrix}
1  \\[2pt]
0
\end{pmatrix},\quad
C = \begin{pmatrix}
0 & 0 \\[2pt]
0 & 0
\end{pmatrix},\nn\\
&D = \begin{pmatrix}
0  \\[2pt]
0
\end{pmatrix},\quad
Q = \begin{pmatrix}
1 & 0 \\[2pt]
0 & 1
\end{pmatrix},\quad
R = 1 ,\quad T = 1.\label{example:no_solution1}
\end{align}
The parameter setting in \eqref{example:no_solution1} satisfies Assumption \ref{ass:H1}, but fails to satisfy the
$L^2$-stabilizability condition stated in Assumption \ref{ass:H2}.
As demonstrated in the Appendix, there exists a positive constant $c$ such that the delayed Riccati equation \eqref{ERiccati}--\eqref{ERiccati-Psi}  does not have a  solution.
\end{example}

Combining Theorem \ref{theorem:verification} and Theorem \ref{theorem:solvability-SERE}, we have the following result.

\begin{theorem}
Let \ref{ass:H1} and \ref{ass:H2} hold.
Then
$$
\bar\Theta_{ES}:=-[R+D^\mathrm{T} \Si(0)D]^{-1}[B^\mathrm{T} \Si(0)+D^\mathrm{T} \Si(0)C],
$$
is an  equilibrium strategy of Problem (LQ-Sub),
where $\Si(\cd)\in C([0,T];\dbS^n_+)$ is a solution to the delayed Riccati equation \rf{ERiccati}--\rf{ERiccati-Psi}.
\end{theorem}

\subsection{Proof of  Theorem \ref{theorem:solvability-SERE}}

To prove Theorem \ref{theorem:solvability-SERE}, we first establish the following semigroup property of matrix-valued stochastic differential equation:

\begin{proposition}\label{lem:sde-semigroup}
Consider the following matrix-valued stochastic differential equation:
\begin{equation}\label{eq:phi-sde}
\begin{cases}
    d\Phi(s) = A\Phi(s)ds + C\Phi(s)dW(s),\qq s\ges 0,\\
    \Phi(0)=I_n,
\end{cases}
\end{equation}
where $A,C\in\dbR^{n\times n}$.
For any constant matrix $X\in\dbR^{n\times n}$, define the linear operator
\begin{equation}
\cT_s(X) = \dbE\bigl[ \Phi(s)^\mathrm{T} X \Phi(s) \bigr],\qquad s\ges 0.
\end{equation}
Then the operator family $\{\cT_s\}_{s\ges 0}$ satisfies:
\begin{enumerate}[1.]
\item $\cT_0(X)=X$ for all matrices $X$;
\item $\cT_{s+r} = \cT_s \circ \cT_r$ for all $s,r\ges 0$;
\item Each $\cT_s(\cd)$ is a bounded linear operator on $\dbR^{n\times n}$ with respect to the Frobenius norm.
\end{enumerate}
Therefore, $\{\cT_s(\cd)\}_{s\ges 0}$ forms a strongly continuous semigroup of bounded linear operators on $\dbR^{n\times n}$, whose infinitesimal generator is
\begin{equation}
\cL(X) := A^\top X + XA + C^\top X C.
\end{equation}
\end{proposition}

The proof of Proposition \ref{lem:sde-semigroup} is given in the Appendix.
Next, we establish a  comparison theorem of algebraic equations.
Recall from Proposition \ref{prop:infinite} that the algebraic Riccati equation associated with the infinite horizon problem \rf{state} and \rf{cost-infty} is given by
\begin{align}
&A^\mathrm{T} P_\infty + P_\infty A + C^\mathrm{T} P_\infty C + Q- (B^\mathrm{T} P_\infty + D^\mathrm{T} P_\infty C)^\mathrm{T} \nn\\
&\qq \times(R + D^\mathrm{T} P_\infty D)^{-1} (B^\mathrm{T} P_\infty + D^\mathrm{T} P_\infty C) = 0. \label{eq:SARE}
\end{align}
Under \ref{ass:H1} and \ref{ass:H2}, the algebraic Riccati equation \rf{eq:SARE} admits a unique static stabilizing solution
$P_\infty\in\dbS^n_+$.
We define the Riccati operator $\mathcal{R}(\cdot)$ for any positive semidefinite matrix $P$ by
\begin{align}
\mathcal{R}(P) &:= A^\mathrm{T} P + PA + C^\mathrm{T} PC + Q- (B^\mathrm{T} P + D^\mathrm{T} P C)^\mathrm{T}\nn\\
&\qq \times (R + D^\mathrm{T} P D)^{-1}
(B^\mathrm{T} P + D^\mathrm{T} P C).
\label{eq:Ric}
\end{align}

\ms
The following \emph{comparison theorem} is the foundation for an a priori estimate of the delayed Riccati equation \rf{ERiccati}--\rf{ERiccati-Psi}.

\begin{lemma}\label{comparison}
For any positive semidefinite matrix $P$ satisfying $\mathcal{R}(P) \ges 0$, we have
$$P \les P_\infty,$$
where $P_\infty\in\dbS^n_+$ is the unique static stabilizing solution to the algebraic Riccati equation \rf{eq:SARE}.
\end{lemma}

\begin{proof}
Denote
\begin{align*}
\Theta_P &= -(R + D^\mathrm{T} P D)^{-1}\big(B^\mathrm{T} P + D^\mathrm{T} P C\big),\\
\Theta_\infty &= -(R + D^\mathrm{T} P_\infty D)^{-1}\big(B^\mathrm{T} P_\infty + D^\mathrm{T} P_\infty C\big).
\end{align*}
For any matrix $Z\in\dbR^{m\times n}$ and positive semidefinite matrix $P\in\dbS^n_+$,
a direct algebraic manipulation yields the following completion-of-squares identity:
\begin{align}
&(A+BZ)^\mathrm{T} P + P(A+BZ) + (C+DZ)^\mathrm{T} P(C+DZ) + Q + Z^\mathrm{T} R Z \notag\\
&= A^\mathrm{T} P+PA+C^\mathrm{T} PC+Q + Z^\mathrm{T}(B^\mathrm{T} P+D^\mathrm{T} PC)+(PB+C^\mathrm{T} PD)Z \nn\\
&\qq+ Z^\mathrm{T}(R+D^\mathrm{T} PD)Z \notag\\
&= \mathcal{R}(P) + (B^\mathrm{T} P+D^\mathrm{T} PC)^\mathrm{T}(R+D^\mathrm{T} PD)^{-1}(B^\mathrm{T} P+D^\mathrm{T} PC) \notag\\
&\qq+ Z^\mathrm{T}(R+D^\mathrm{T} PD)Z + Z^\mathrm{T}(B^\mathrm{T} P+D^\mathrm{T} PC)+(PB+C^\mathrm{T} PD)Z \notag\\
&=\mathcal{R}(P) +\Theta_P^\mathrm{T}(R+D^\mathrm{T} PD)\Theta_P \notag+ Z^\mathrm{T}(R+D^\mathrm{T} PD)Z - Z^\mathrm{T} (R + D^\mathrm{T} P D)\Theta_P\\
&\qq -\Theta_P^\mathrm{T} \big(R + D^\mathrm{T} P D\big) Z \notag\\
&= \mathcal{R}(P) + (Z-\Theta_P)^\mathrm{T}\big(R + D^\mathrm{T} P D\big)(Z-\Theta_P).
\label{eq:squares}
\end{align}
Taking $Z=\Theta_\infty$ in \eqref{eq:squares}, we have
\begin{equation}\label{eq:Theta_infty}
\begin{aligned}
&(A+B\Theta_\infty)^\mathrm{T} P + P(A+B\Theta_\infty) + (C+D\Theta_\infty)^\mathrm{T} P(C+D\Theta_\infty) + Q + \Theta_\infty^\mathrm{T} R \Theta_\infty \\
& = \mathcal{R}(P) + (\Theta_\infty-\Theta_P)^\mathrm{T} (R + D^\mathrm{T} P D)(\Theta_\infty - \Theta_P) \ges 0,
\end{aligned}
\end{equation}
where the inequality is due to $\cR(P)\ges 0$.
Since $P_\infty$ solves $\mathcal{R}(P_\infty)=0$, taking $P=P_\infty$ and $Z=\Theta_\infty$ in \eqref{eq:squares} yields
\begin{equation}
(A+B\Theta_\infty)^\mathrm{T} P_\infty + P_\infty(A+B\Theta_\infty)
+ (C+D\Theta_\infty)^\mathrm{T} P_\infty(C+D\Theta_\infty)
+ Q + \Theta_\infty^\mathrm{T} R \Theta_\infty = 0.
\label{eq:Pinf_squares}
\end{equation}
Define the difference matrix $\Delta = P_\infty - P$.
Subtracting the above two equations \eqref{eq:Theta_infty} and \eqref{eq:Pinf_squares},
we have the Lyapunov-type equation for $\Delta$:
\begin{align*}
&(A+B\Theta_\infty)^\mathrm{T} \Delta + \Delta (A+B\Theta_\infty)
+ (C+D\Theta_\infty)^\mathrm{T} \Delta (C+D\Theta_\infty)\nn \\
& = -\mathcal{R}(P) - (\Theta_\infty - \Theta_P)^\mathrm{T} (R + D^\mathrm{T} P D)(\Theta_\infty - \Theta_P) \les 0.
\label{eq:Delta_eq}
\end{align*}
Since the system $[A+B\Theta_\infty,\,C+D\Theta_\infty]$ (i.e., the closed-loop system \rf{prop:infinite1}) is $L^2$-stable, we conclude that $\Delta \ges 0$ by \cite[Proposition 3.5]{Huang-2015} or  \cite[Lemma 2.2]{Sun-Yong-2018}.
\end{proof}

With Lemma \ref{comparison}, we have the following a priori estimate for the delayed Riccati equation \rf{ERiccati}--\rf{ERiccati-Psi}.

\begin{proposition}\label{bound-Gamma}
Let \autoref{ass:H1} and \autoref{ass:H2} hold. If the delayed Riccati equation \rf{ERiccati}--\rf{ERiccati-Psi} admits a solution $\Si(\cdot)$, then $0 \les \Si(0) \les P_\infty$, where $P_\infty\in\dbS^n_+$ is the unique static stabilizing solution to the algebraic Riccati equation \rf{eq:SARE}.
\end{proposition}

\begin{remark}
We remark that the above a priori estimate of $\Si(0)$ is independent of $T$.
From Theorem \ref{thm:limit}, which will be given later, we can also see that this estimate is sharp.

\end{remark}

\begin{proof}
Let $\Phi_{\Si}(\cd)$ be the solution to the following stochastic differential equation:
\begin{equation*}
\begin{cases}
d \Phi_{\Si}(t) =[A + B\Psi]\Phi_{\Si}(t)dt + [C + D\Psi]\Phi_{\Si}(t)d W(t),\\[4pt]
\Phi_{\Si}(0) = I,
\end{cases}
\end{equation*}
where $\Psi$ is defined by \rf{ERiccati-Psi}.
 Define the linear operator
\begin{equation*}
\cT_t(X) = \dbE\bigl[ \Phi_{\Si}(t)^\mathrm{T} X \Phi_{\Si}(t) \bigr],\qq \forall X\in\dbR^{n\times n}.
\end{equation*}

By Lemma \ref{lem:sde-semigroup}, $\{\cT_t\}_{t\ge 0}$ forms a strongly continuous semigroup on $\dbR^{n\times n}$, whose infinitesimal generator is given by $\cL(X) = (A+B\Psi)^\mathrm{T} X + X(A+B\Psi) + (C+D\Psi)^\mathrm{T} X (C+D\Psi)$.
Note that the solution $\Si(\cdot)$ to equation \eqref{ERiccati} satisfies
\begin{align*}
\Si(0)& =\dbE\int_0^T \Phi_\Si(r)^\mathrm{T} (Q + \Psi^\mathrm{T} R\Psi) \Phi_\Si(r)dr\\
&= \int_0^T \cT_{r}\bigl(Q +\Psi^\mathrm{T} R\Psi\bigr) dr \ges 0.
\end{align*}
Applying the generator $\cL$ to both sides, we obtain
\begin{equation*}
\cL\bigl(\Si(0)\bigr) = \int_0^T \cL\bigl(\cT_{r}\bigl(Q +\Psi^\mathrm{T} R\Psi\bigr)\bigr) dr.
\end{equation*}
Since $\cL$ is the infinitesimal generator of $\{\cT_t\}$, it holds that $\cL\circ\cT_{r} = \frac{d}{dr}\cT_{r}$. Thus,
\begin{equation}\label{eq:Gamma0-L}
\cL\bigl(\Si(0)\bigr) = \int_0^T \frac{d}{dr}\cT_{r}\bigl(Q +\Psi^\mathrm{T} R\Psi\bigr) dr
= \cT_T\bigl(Q +\Psi^\mathrm{T} R\Psi\bigr) -(Q +\Psi^\mathrm{T} R\Psi).
\end{equation}
From \eqref{eq:Ric} and \rf{eq:Gamma0-L}, we have
\begin{align*}
\mathcal{R}\bigl(\Si(0) \bigr) &=
 A^\mathrm{T} \Si(0) + \Si(0) A + C^\mathrm{T} \Si(0) C + Q\nn\\
 &\qq - (B^\mathrm{T} \Si(0)  + D^\mathrm{T} \Si(0)  C)^\mathrm{T} (R + D^\mathrm{T} \Si(0)  D)^{-1}
(B^\mathrm{T} \Si(0)  + D^\mathrm{T} \Si(0)  C)\nn\\
&=  A^\mathrm{T} \Si(0) + \Si(0) A + C^\mathrm{T} \Si(0) C + Q - \Psi^\mathrm{T} (R + D^\mathrm{T} \Si(0) D) \Psi\nn \\
&=  (A+B\Psi)^\mathrm{T} \Si(0) + \Si(0)(A+B\Psi) + (C+D\Psi)^\mathrm{T} \Si(0) (C+D\Psi)\nn \\
&\qq+ Q + \Psi^\mathrm{T} R\Psi \nn\\
&= \cL\bigl(\Si(0)\bigr) + Q + \Psi^\mathrm{T} R\Psi  \nn\\
&= \cT_T\bigl(Q + \Psi^\mathrm{T} R\Psi\bigr) \ges 0.
\label{eq:residual-Gamma}
\end{align*}
By Lemma \ref{comparison}, we immediately conclude $\Si(0) \les P_\infty$.
This completes the proof.
\end{proof}

For any given $\Si\in\dbS^n_+$, consider the following Lyapunov equation:
\bel{Lya-E}\left\{
\begin{aligned}
&\dot{\Si}(s)+\Si(s)[A+B\Psi]+[A+B\Psi]^\mathrm{T}\Si(s)+[C+D\Psi]^\mathrm{T}\Si(s)[C+D\Psi]\\
&\qq+[Q+\Psi^\mathrm{T} R\Psi]=0,\qq 0\les s\les T,\\
&\Si(T)=0,
\end{aligned}\right.
\ee
with
\bel{}
\Psi=-[R+D^\mathrm{T} \Si D]^{-1}[B^\mathrm{T} \Si +D^\mathrm{T} \Si C].
\ee
Clearly, it admits a unique positive semidefinite solution $\Si(\cd)\in\dbS^n_+$.
Thus, the following mapping is well-defined:
\bel{mapping}
\Pi(\Si)=\Si(0),\qq \forall \Si\in\dbS^n_+.
\ee
Since $\mathbb{S}_+^n$ is a set but not a linear space, to apply the Leray--Schauder fixed point theorem (i.e., Lemma \ref{lem:leray-schauder}), we introduce the following decomposition method.

\begin{definition}
Let $\Si$ be a real symmetric matrix. Then $\Si$ can be uniquely decomposed as
$$
\Si =\Si_+ + \Si_-,
$$
where $\Si_+$ is positive semidefinite and $\Si_-$ is negative semidefinite.
If $\Si = F\Lambda F^\mathrm{T}$ is the spectral decomposition with $\Lambda = \operatorname{diag}(\lambda_1,\dots,\lambda_n)$, then
$$
\Si_+ = F \Lambda_+ F^\mathrm{T}, \qquad \Si_- = F \Lambda_- F^\mathrm{T},
$$
with $(\Lambda_+)_{ii} = \max(\lambda_i,0)$ and $(\Lambda_-)_{ii} = \min(\lambda_i,0)$.
\end{definition}

\ms
Finally, we are ready to present the proof of Theorem \ref{theorem:solvability-SERE}.

\begin{proof}[Proof of Theorem \ref{theorem:solvability-SERE}]
We now apply the Leray--Schauder fixed point theorem (i.e., Lemma \ref{lem:leray-schauder}) to prove the existence of a solution to the delayed Riccati equation \rf{ERiccati}--\rf{ERiccati-Psi}.
Define the Banach space $X = \dbS^n$, where $\dbS^n$ is the space of $n\times n$ symmetric matrices, equipped with the Frobenius norm $\|\Si\|_F$.

\ms

For any $\Si\in X$ and $\sigma\in[0,1]$, define the mapping
\begin{equation*}
f(\Si,\sigma) = \Si(0),
\end{equation*}
where $\Si(\cdot)\in\dbS^n_+$ satisfies the Lyapunov equation:
\bel{Riccati-Gamma-P}\left\{
\begin{aligned}
&\dot{\Si}(s)+\Si(s)[A+B\Psi]+[A+B\Psi]^\mathrm{T}\Si(s)+[C+D\Psi]^\mathrm{T}\Si(s)[C+D\Psi]\\
&\qq+[Q+\Psi^\mathrm{T} R\Psi]=0,\qq 0\les s\les \sigma T,\\
&\Si(\si T)=0,
\end{aligned}\right.
\ee
with
\bel{Gamma-strategy-P}
\Psi=-[R+D^\mathrm{T} \Si_+ D]^{-1}[B^\mathrm{T} \Si +D^\mathrm{T} \Si C].
\ee
Note that in \rf{Gamma-strategy-P}, the singular term is given by $\Si_+$.
Under Assumption~\ref{ass:H1}, $R>0$ together with $\Si_+\ges 0$ implies that $R+D^\mathrm{T}\Si_+D \ges R>0$. Hence its inverse is uniformly bounded for bounded $\Si$, which in turn yields a uniform bound on $\Psi$ in \eqref{Gamma-strategy-P}. Consequently, all coefficients in the Lyapunov equation \eqref{Riccati-Gamma-P} are bounded uniformly with respect to $\Si$ and $\sigma$. Since the equation is linear and the interval $[0,\sigma T]\subseteq[0,T]$ is finite, its solution $\Si(\cdot)$ is uniformly bounded; in particular, $\Si(0)=f(\Si,\sigma)$ is bounded. Thus $f$ maps bounded sets into bounded sets and, as $X$ is finite-dimensional, is compact. Moreover, $f(\Si,0)=0$ for all $\Si\in X$ follows directly from the terminal condition $\Si(0)=0$ in \eqref{Riccati-Gamma-P}.

\ms
Assume that $f(\Si,\sigma) = \Si$ for some $\Si\in X$ and $\sigma\in[0,1]$. Then $\Si$ is positive semidefinite,
which implies that $\Si_+ =\Si$.
Thus, $\Si = \Si(0)$, where $\Si(\cd)$ satisfies \eqref{Riccati-Gamma-P} with $\Psi$ replaced by
\bel{Gamma-strategy-P*}
\Psi=-[R+D^\mathrm{T} \Si D]^{-1}[B^\mathrm{T} \Si +D^\mathrm{T} \Si C].
\ee
Then by \autoref{bound-Gamma}, there exists a constant $r>0$ such that $\{ \Si \in X;\ f(\Si,\sigma) = \Si \text{ for some } 0 \le \sigma \le 1 \} \subset B(0;r)$. Actually, we can take $r = \|P_\infty\|_F + 1$.
By the Leray--Schauder fixed point theorem, the mapping $f(\cdot,1)$ has at least one fixed point $\Si^*\in\dbS^n_+$ in $\overline{B(0;r)}$.
Then the solution $\Si(\cd)\in\dbS^n_+$ to \rf{Riccati-Gamma-P}
with $\Psi=-[R+D^\mathrm{T} \Si^* D]^{-1}[B^\mathrm{T} \Si^* +D^\mathrm{T} \Si^* C]$ and $\si=1$
is a solution to the delayed Riccati equation \rf{ERiccati}--\rf{ERiccati-Psi}.
\end{proof}

\section{The convergence behavior as $T\to \infty$}\label{sec:con}

Let $\Si_T(\cd)$ be a positive semidefinite solution to the delayed Riccati equation \rf{ERiccati}--\rf{ERiccati-Psi} with the terminal time $T$.
In this section, we shall investigate the limiting behavior of $\Si_T(\cd)$ as $T\to\infty$.

\begin{theorem}\label{thm:limit}
Let \ref{ass:H1} and \ref{ass:H2} hold. Then
\begin{equation}\label{thm:limit1}
\lim_{T\to\infty} \Si_T(0) = P_\infty,
\end{equation}
where $P_\infty\in\dbS^n_+$ is the unique   static stabilizing solution to
the algebraic Riccati equation \rf{ARE-infty}. Moreover, the equilibrium $\bar\Theta_{ES}:=\Psi$, where $\Psi$ is given by \rf{ERiccati-Psi}, has the following convergence property:
\begin{equation}\label{thm:limit2}
\lim_{T\to\infty} \bar\Theta_{ES} =-  [R+D^\mathrm{T} P_\infty D]^{-1}[B^\mathrm{T} P_\infty+D^\mathrm{T} P_\infty C]=\Th_\infty,
\end{equation}
where $\Th_\infty$ is the optimal feedback strategy of the infinite horizon problem.
\end{theorem}

\begin{proof}
For every $T>0$, by Proposition \ref{bound-Gamma}, we have
$$
0 \les \Si_T(0) \les P_\infty.
$$
By the Bolzano--Weierstrass theorem, there exists a $\Si_\infty\in\dbS^n_+$ and a sequence $T_k\to\infty$ such that
\begin{equation*}
\lim_{k\to\infty} \Si_{T_k}(0) =\Si_\infty\quad\text{with }0\les \Si_\infty\les P_\infty.
\end{equation*}
Set
\begin{equation*}
\begin{cases}
d \Phi_k(s) = [A+B\Psi_k] \Phi_k(s)ds + [C+D\Psi_k] \Phi_k(s)dW(s),\qq s\in[0,T_k],\\
\Phi_k(0) = I,\\
\Psi_k = -[R+D^\mathrm{T} \Si_{T_k}(0) D]^{-1}[B^\mathrm{T} \Si_{T_k}(0) +D^\mathrm{T}\Si_{T_k}(0) C].
\end{cases}
\end{equation*}
%
For any fixed $S>0$ and sufficiently large $k$ with $T_k>S$, by the integral representation of $\Si_{T_k}(0)$ we have
\begin{align*}
\Si_{T_k}(0)&=\dbE \int_0^{T_k} \Bigl[\Phi_k(s)^\mathrm{T} (Q+\Psi_k^\mathrm{T} R \Psi_k) \Phi_k(s)\Bigr]ds \notag \\
&\ges \dbE\int_0^{S} \Bigl[\Phi_k(s)^\mathrm{T} (Q+\Psi_k^\mathrm{T} R \Psi_k) \Phi_k(s)\Bigr]ds.
\end{align*}
Taking the limit $k\to\infty$ and applying the dominated convergence theorem, we obtain
\begin{equation*}
\Si_\infty \ges \dbE\int_0^S \Bigl[\Phi_\infty (s)^\mathrm{T} (Q+\Psi_\infty^\mathrm{T} R \Psi_\infty) \Phi_\infty(s)\Bigr]ds,
\end{equation*}
where $\Phi_\infty$ is the unique solution to the following closed-loop system:
\begin{equation*}
\begin{cases}
d \Phi_\infty(s) = [A+B\Psi_\infty] \Phi_\infty(s)ds + [C+D\Psi_\infty] \Phi_\infty(s)dW(s),\qq s\in[0,\infty),\\
\Phi_\infty(0) = I,
\end{cases}
\end{equation*}
with
\bel{thm:limit-p1}
\Psi_\infty = -[R+D^\mathrm{T} \Si_\infty D]^{-1}[B^\mathrm{T} \Si_\infty +D^\mathrm{T} \Si_\infty C].
\ee
Since $S$ is arbitrary, we conclude that
\begin{equation}\label{thm:limit-p2}
\dbE\int_0^\infty\Bigl[\Phi_\infty (s)^\mathrm{T} (Q+\Psi_\infty^\mathrm{T} R \Psi_\infty) \Phi_\infty(s)\Bigr]ds \les \Si_\infty \les P_\infty.
\end{equation}
Moreover, for any $x\in\dbR^n$, the unique solution of the following closed-loop system is given by $X(\cd)=\Psi_\infty(\cd)x$:
\begin{equation*}
\begin{cases}
d X(s)= [A+B\Psi_\infty] X(s)ds + [C+D\Psi_\infty] X(s)dW(s),\\
X(0)=x.
\end{cases}
\end{equation*}
Thus, from \rf{thm:limit-p2} we have
$$
\d \dbE\int_0^\infty |X(s)|^2ds \les \dbE\int_0^\infty\Bigl[x^\mathrm{T}\Phi_\infty (s)^\mathrm{T} (Q+\Psi_\infty^\mathrm{T} R \Psi_\infty ) \Phi_\infty(s)x\Bigr]ds \les \lan \Si_\infty x,x\ran \les \lan P_\infty x,x\ran,
$$
for some $\d>0$, in which the first inequality is due to \ref{ass:H1'}.
This means that the strategy $\Psi_\infty$, defined by \rf{thm:limit-p1}, is a stabilizer of the system $[A,B;C,D]$.
Thus, the control process $u_\infty(\cd):=\Psi_\infty X(\cd)$ belongs to $\sU_{ad}(0,x)$.
Recall from Proposition \ref{prop:infinite} that $\lan P_\infty x,x\ran$ is the value function. Then we have
\begin{align*}
\lan P_\infty x,x\ran&=\inf_{u(\cd)\in\sU_{ad}(0,x)}\cJ_\infty(0,x;u(\cd))\les \cJ_\infty(0,x;u_\infty(\cd)) \\
&= \dbE\int_0^\infty \Bigl[[\Phi_\infty (s)x]^\mathrm{T} (Q+\Psi_\infty^\mathrm{T} R \Psi_\infty) [\Phi_\infty(s)x]\Bigr]ds \\
&\les \lan\Si_\infty x,x\ran\les\lan P_\infty x,x\ran, \qq \forall x\in\dbR^n,
\end{align*}
which implies
\begin{equation*}
\Si_\infty=P_\infty.
\end{equation*}
Since $\{\Si_T(0)\}_{T>0}$ is compact and the limit of any convergent subsequence of $\{\Si_T(0)\}_{T>0}$ equals $P_\infty$, we obtain $\lim\limits_{T\to\infty} \Si_T(0) = P_\infty$.
The convergence \rf{thm:limit2} can be obtained easily.
\end{proof}

\section{Concluding}\label{sec:conclu}
The main contribution of this paper is that we provide a constructive method for the equilibrium strategy of Problem (LQ-Sub).
Interestingly, the equilibrium strategy retains the important time-invariance property of optimal feedback strategies
in infinite horizon optimal control problems. The main mathematical innovation is that the equilibrium Riccati equation associated with Problem (LQ-Sub)
turns out to be a delayed backward Riccati equation, which is interesting in its own right.
Being essentially a Fredholm integral equation,
this type of equation poses a fundamental challenge to establishing its solvability.
We overcome this difficulty by introducing some new ideas.
We believe that similar ideas can be applied to other problems,
such as sub-infinite horizon nonlinear control problems, infinite horizon time-inconsistent control problems, among others.
We will report the related results in the near future.

\section*{Appendix}

\subsection*{Proof of Proposition \ref{lem:sde-semigroup}}
Recall that $\dbF\equiv\{\cF_s\}_{s\ges0}$ is the natural filtration generated by the Brownian motion $W(\cd)$.
For fixed $s\ges 0$, define $\tilde W(r) := W(s+r)-W(s)$ with $r\ges 0$. By the independent increment property of Brownian motion, $\tilde W(\cd)$ is a standard Brownian motion independent of $\cF_s$.

\ms
Let $\tilde\Phi(\cd)$ solve the stochastic differential equation
$$
\left\{
\begin{aligned}
d\tilde\Phi(r) &= A\tilde\Phi(r)dr + C\tilde\Phi(r)d\tilde W(r),\qq r\ges 0,\\
 \tilde\Phi(0)&=I.
\end{aligned}
\right.
$$
It follows that $\tilde\Phi(r)$ is  independent of $\Phi(s)$, and identically distributed with $\Phi(r)$.
Define the process $Z(r) = \tilde\Phi(r-s)\Phi(s)$ for $r\ges s$.
Clearly, $Z(r)$ is adapted to $\cF_r$ and satisfies
\begin{equation}\label{eq:Z-sde}
\left\{\begin{aligned}
dZ(r)& = A Z(r)dr + C Z(r)d W(r),\qq r\ges s,\\
 Z(s)&=\Phi(s).
\end{aligned}\right.
\end{equation}
By the strong uniqueness of  solutions to \rf{eq:Z-sde}, we obtain $Z(r)=\Phi(r)$ almost surely.
This yields the decomposition
\bel{lem:sde-semigroup-p1}
\Phi(s+r)=\ti\Phi(r)\Phi(s),\qq \forall s,r\ges 0.
\ee

(1) Since $\Phi(0)=I$, we have
\begin{equation*}
\cT_0(X)=\dbE[I^\mathrm{T} X I]=X.
\end{equation*}

(2) From \rf{lem:sde-semigroup-p1}, we have
\begin{equation*}
\Phi(s+r)^\mathrm{T} X \Phi(s+r) = \Phi(s)^\mathrm{T} \bigl( \tilde\Phi(r)^\mathrm{T} X \tilde\Phi(r) \bigr) \Phi(s),\qq \forall s,r\ges 0.
\end{equation*}
Taking expectation and using the independence of $\tilde\Phi(r)$ and $\Phi(s)$, we obtain
\begin{align*}
\cT_{s+r}(X)
&= \dbE\big[\Phi(s+r)^\mathrm{T} X\Phi(s+r)\big] \\
&= \dbE\big[\Phi(s)^\mathrm{T} \tilde\Phi(r)^\mathrm{T} X\tilde\Phi(r)\Phi(s)\big] \\
&= \dbE\Big[\Phi(s)^\mathrm{T} \dbE\big[\tilde\Phi(r)^\mathrm{T} X\tilde\Phi(r)\,\big|\,\cF_s\big]\Phi(s)\Big] \\
&= \dbE\big[\Phi(s)^\mathrm{T} \cT_r(X)\Phi(s)\big] \\
&= \cT_s\big(\cT_r(X)\big).
\end{align*}

(3) Direct estimation yields
\begin{equation*}
   \|\cT_s(X)\|_F^2 \les \dbE\bigl[ \|\Phi(s)^\mathrm{T} X \Phi(s)\|_F^2 \bigr] \les \dbE\bigl[ \|\Phi(s)\|_F^4 \bigr] \|X\|_F^2.
\end{equation*}
Since all moments of linear SDE solutions are bounded on finite time intervals, we have $\|\cT_s\| < \infty$.

\ms
We now derive the explicit form of the infinitesimal generator $\cL(\cd)$ by applying  It\^o's formula.
For any fixed matrix $X\in\mathbb{R}^{n\times n}$, consider the process $\Phi(s)^\mathrm{T} X\Phi(s)$.
The It\^o differential of this matrix-valued process is given by
\begin{align*}
&d\bigl(\Phi(s)^\mathrm{T} X \Phi(s)\bigr)
= \bigl(d\Phi(s)\bigr)^\mathrm{T} X \Phi(s)
+ \Phi(s)^\mathrm{T} X d\Phi(s)
+ \bigl(d\Phi(s)\bigr)^\mathrm{T} X d\Phi(s)\\
&= \Phi(s)^\mathrm{T} \bigl(A^\mathrm{T} X + XA + C^\mathrm{T} X C\bigr)\Phi(s)\,ds
+ \Phi(s)^\mathrm{T} \bigl(C^\mathrm{T} X + X C\bigr)\Phi(s)\,dW(s).
\end{align*}
Taking expectation on both sides, we have
\begin{equation*}
\dbE\bigl[\Phi(s)^\mathrm{T} X \Phi(s)\bigr]
= \dbE\bigl[\Phi(0)^\mathrm{T} X \Phi(0)\bigr]
+ \int_0^s \dbE\Bigl[\Phi(r)^\mathrm{T} \bigl(A^\mathrm{T} X + XA + C^\mathrm{T} X C\bigr)\Phi(r)\Bigr]dr.
\end{equation*}
Differentiating both sides with respect to $s$ and recalling $\cT_s(X) = \dbE\big[\Phi(s)^\mathrm{T} X\Phi(s)\big]$, we arrive at
\begin{equation*}
\frac{d}{ds}\cT_s(X)
= \dbE\Bigl[\Phi(s)^\mathrm{T} \bigl(A^\mathrm{T} X + XA + C^\mathrm{T} X C\bigr)\Phi(s)\Bigr].
\end{equation*}
By definition, the infinitesimal generator satisfies $\cL(X) = \displaystyle\frac{d}{ds}\cT_s(X)\big|_{s=0}$, which gives
$$
\cL(X) = A^\mathrm{T} X + XA + C^\mathrm{T} X C.
$$
This completes the proof.

\subsection*{Proof of Example \ref{example:no_solution}}
If the delayed Riccati equation \rf{ERiccati}--\rf{ERiccati-Psi} with the setting \rf{example:no_solution1} admits a solution $\Si(\cdot)$, then $\Si(s)$ is symmetric for any $s\ges 0$. Denote
$$
\Si(0) = \begin{pmatrix} x & y \\ y & z \end{pmatrix}.
$$
Direct calculation yields
$$
A_c+B\Psi
= \begin{pmatrix} 1-x & 1-y \\ 0 & c \end{pmatrix},\qquad
Q + \Psi^\mathrm{T} R\Psi
= \begin{pmatrix} 1+x^2 & xy \\ xy & 1+y^2 \end{pmatrix}.
$$
It will be shown that $1-x<0$. Since $c>0$, we have $c\neq 1-x$, so the matrix exponential can be expressed explicitly as
$$
e^{(A_c+B\Psi)t}
= \begin{pmatrix}
e^{(1-x) t} & \displaystyle (1-y)\,\frac{e^{ct} - e^{(1-x) t}}{c-(1-x)} \\
0 & e^{ct}
\end{pmatrix}.
$$
Since $C\equiv 0$ and $D\equiv 0$, direct integration of \rf{ERiccati}--\rf{ERiccati-Psi} over the interval $[0,1]$ yields the integral equation:
\begin{equation}\label{eq:inteq}
\Si(0)
= \int_0^1 e^{\,(A_c + B \Psi)^\mathrm{T} t}
\bigl(Q + \Psi^\mathrm{T} R\Psi\bigr)
e^{\,(A_c +B \Psi) t} dt.
\end{equation}
Extracting the $(1,1)$-entry from the above integral equation, we obtain the following scalar relation:
\bel{example:no_solution-p7}
x = (1+x^2) \int_0^1 e^{2(1-x) t} dt.
\ee
The right-hand side is strictly positive, which immediately implies $x>0$. If $0<x\les 1$, then $2(1-x)\ges 0$, so the integral $\displaystyle\int_0^1 e^{2(1-x) t} dt$ is bounded below by $1$, and hence
$$
(1+x^2)\int_0^1 e^{2(1-x) t} dt > x.
$$
Thus no solution exists for $0<x\les 1$, and we must have $x>1$.

\medskip

Since $x>1$, we have $x-1>0$. Directly integrating the right-hand side of equation \rf{example:no_solution-p7} yields
$$
\int_0^1 e^{2(1-x) t} dt
= \frac{1 - e^{2(1-x)}}{2(x-1)}.
$$
Substituting the integral result back gives
$$
2(x-1)x = \big(x^2+1\big)\big(1 - e^{-2(x-1)}\big).
$$
After simplification, this reduces to
\bel{example:no_solution-p6}
e^{-2(x-1)} = \frac{-x^2 + 2x + 1}{x^2+1}.
\ee
Clearly, the above equation \rf{example:no_solution-p6} can only have a solution over $1<x<1+\sqrt{2}$.
Define the equation
$$
\psi(x) := \ln\left(\frac{x^2+1}{-x^2 + 2x + 1}\right) - 2(x-1) = 0,\qquad 1<x<1+\sqrt{2}.
$$
Then the solvability of \rf{example:no_solution-p6} over $1<x<1+\sqrt{2}$ is equivalent to that of the above equation.
It can be verified that $\psi(1)=0$, $\psi'(1)=-1$, and $\psi(x)\to +\infty$ as $x\uparrow 1+\sqrt{2}$. The second derivative is explicitly computed as
$$
\psi''(x)
= \frac{4(x-1)^5+28(x-1)^4+48(x-1)^3+32(x-1)^2+16(x-1)+16}
{\big(2-(x-1)^2\big)^2\big((x-1)^2+2(x-1)+2\big)^2} > 0,
$$
which implies that $\psi(\cdot)$ is strictly convex on $(1,1+\sqrt{2})$.
Therefore, the equation admits exactly one positive root. Numerical evaluation gives
$$
x \approx 2.217136713648735,\qquad
1-x \approx -1.217136713648735.
$$

With the above fixed $x$ and $1-x$, the $(1,2)$-entry of the integral equation \rf{eq:inteq} yields
\bel{example:no_solution-p1}
y = \int_0^1 e^{(1-x) t}
\left[
(1+x^2)\,(1-y)\,\frac{e^{ct}-e^{(1-x) t}}{c-(1-x)}
+\, x y\, e^{ct}
\right] dt.
\ee
Denote
$$
K_1(c) = (1+x^2)\int_0^1 e^{(1-x) t}\,\frac{e^{ct}-e^{(1-x) t}}{c-(1-x)}dt,\qquad
K_2(c) = x\int_0^1 e^{((1-x)+c)t}dt.
$$
Then the equation \rf{example:no_solution-p1} can be simplified to
$$
y = (1-y)K_1(c) + y K_2(c),
$$
or equivalently
\bel{example:no_solution-p10}
y\big(1 + K_1(c) - K_2(c)\big) = K_1(c).
\ee
Denote
$$
F(c) = 1 + K_1(c) - K_2(c).
$$
We claim that there exists some $c>1$ such that $F(c)=0$. Direct substitution of $c = 1$ into the integral definitions of $K_1(c)$ and $K_2(c)$, together with the valid identity $x = (1+x^2)\int_0^1 e^{2\lambda t}\,dt$ for the fixed $\lambda = 1 - x$, yields a positive value
$$
F(1) > 0.
$$
As $c\to +\infty$, we have the asymptotic expansions
$$
K_1(c) \sim \frac{1+x^2}{c^2}e^{\lambda+c},\qquad
K_2(c) \sim \frac{x}{c}e^{\lambda+c},
$$
which imply that $K_2(\cd)$ dominates $K_1(\cd)$ asymptotically. Thus, we have $F(c)\to -\infty$ as $c\to +\infty$. By the continuity of $F(\cdot)$ over $c>0$ and the intermediate value theorem, there exists a constant $c_*>1$ such that $F(c_*)=0$. Numerical computation gives
$$
c_* \approx 1.695698231787353.
$$

Now fix $c = c_*$. Since $c_*>1>\lambda$, we have
$$
\frac{e^{c_* t}-e^{\lambda t}}{c_*-\lambda} > 0,\quad \forall t>0,
$$
which implies $K_1(c_*)>0$. Substituting $F(c_*)=0$ into the algebraic equation \rf{example:no_solution-p10} leads to
$$
0\cdot y = K_1(c_*).
$$
This is a clear contradiction. Consequently, the system \eqref{ERiccati}--\eqref{ERiccati-Psi} does not have a solution for $c = c_*$.

\end{document}